\documentclass[12pt]{article}
\usepackage[utf8]{inputenc}
\usepackage{amsmath}
\usepackage{amsfonts}
\usepackage{amssymb}
\usepackage{amsthm}
\newtheorem{thm}{Theorem}
\newtheorem*{thm*}{Theorem}
\newtheorem{lem}{Lemma}
\newtheorem{cnj}{Conjecture}
\newtheorem{qst}{Question}
\newtheorem{obs}{Observation}
\newtheorem{defn}{Definition}
\title{Graph Operations and Upper Bounds on Graph Homomorphism Counts}
\author{Luke Sernau}
\begin{document}
\maketitle
\begin{abstract}
We construct a family of countexamples to a conjecture of Galvin \cite{mainConjReference}, which stated that for any $n$-vertex, $d$-regular graph $G$ and any graph $H$ (possibly with loops),
\[\hom(G,H) \leq \max\left\lbrace\hom(K_{d,d}, H)^{\frac{n}{2d}}, \hom(K_{d+1},H)^{\frac{n}{d+1}}\right\rbrace,\]
where $\hom(G,H)$ is the number of homomorphisms from $G$ to $H$.

By exploiting properties of the graph tensor product and graph exponentiation, we also find new infinite families of $H$ for which the bound stated above on $\hom(G,H)$ holds for all $n$-vertex, $d$-regular $G$.

In particular we show that if $H_{\rm WR}$ is the complete looped path on three vertices, also known as the Widom-Rowlinson graph, then 
$$
{\hom}(G,H_{\rm WR}) \leq {\hom}(K_{d+1},H_{\rm WR})^\frac{n}{d+1}
$$ 
for all $n$-vertex, $d$-regular $G$. This verifies a conjecture of Galvin.
\end{abstract}
\section{Introduction}
Graph homomorphisms are an important concept in many areas of graph theory. A graph homomorphism is simply an adjacency-preserving map between a graph $G$ and a graph $H$. That is, for given graphs $G$ and $H$, a function $\phi: V(G) \rightarrow V(H)$ is said to be a homomorphism from $G$ to $H$ if for every edge $uv \in E(G)$, we have $\phi(u)\phi(v) \in E(H)$ (Here, as throughout, all graphs are simple, meaning without multi-edges, but they are permitted to have loops). When $H=K_q$, the complete loop-free graph on $q$ vertices, a homomorpism from $G$ to $H$ is just a proper $q$-coloring of $G$. For this reason homomorphisms from $G$ to $H$ are often referred to as $H$-colorings of $G$.

While graph homomorphisms have proven to be a very useful tool, many extremal problems involving them remain open \cite{cutler}. One such question is the following.

\begin{qst} \label{mainQuestion}
Let $H$ be an arbitrary graph. Fix natural numbers $d$ and $n$, with $d < n$. What $d$-regular graphs $G$ on $n$ vertices maximize $\hom(G,H)$, the number of homomorphisms from $G$ to $H$?
\end{qst}

This question remains open, but one conjecture (first given in \cite{mainConjReference}) is that the maximizer graph for any $H$ is one of two possibilities.

\begin{cnj}\label{fullcnj}
Fix an arbitrary graph $H$. Let $G$ be a loop-free $d$-regular graph on $n$ vertices. Then
\[\hom(G,H) \leq \max\left\lbrace\hom(K_{d,d}, H)^{\frac{n}{2d}}, \hom(K_{d+1},H)^{\frac{n}{d+1}}\right\rbrace.\]
\end{cnj}

Note that if $2d$ divides $n$,
\[\hom(K_{d,d}, H)^{\frac{n}{2d}} = \hom\left(\frac{n}{2d}K_{d,d},H\right),\]
where by $\frac{n}{2d}K_{d,d}$ we mean the disjoint union of $\frac{n}{2d}$ copies of $K_{d,d}$. Likewise, if $d+1$ divides $n$,
\[\hom(K_{d+1}, H)^{\frac{n}{d+1}} = \hom\left(\frac{n}{d+1}K_{d+1},H\right).\]

So Conjecture \ref{fullcnj} is really making a claim about maximizer graphs, rather than merely asserting an upper bound. Namely, it asserts that up to some divisibility considerations, the maximizer of the number of homomorphisms is either the disjoint union of an appropriate number of copies of $K_{d,d}$ or the disjoint union of an appropriate number of copies of $K_{d+1}$. This conjecture arose as a correction to an earlier conjecture (made in \cite{bipartiteReference}), that $\hom(G,H) \leq \hom(K_{d,d}, H)^{\frac{n}{2d}}$, always. In \cite{bipartiteReference}, Galvin and Tetali were able to show this stronger form for the special case where $G$ is bipartite, and so it was natural to consider generalizing the result to all graphs. But taking $H$ to be the disjoint union of loops gives a case where disjoint copies of $K_{d+1}$ is the maximizer, and so we are forced to consider the weaker form given in Conjecture \ref{fullcnj}.

The first aim of this article is to show that Conjecture \ref{fullcnj} is false. Indeed, in Section \ref{counterexampleSection} we give counterexamples for all $d \geq 4$ (see Theorem \ref{counterexampleMain} for the exact statement).

It is now not clear what the correct version of Conjecture \ref{fullcnj} should be for general $H$. In the absence of a general statement, one sensible avenue of approach is to consider conditions on $H$ under which one of $\frac{n}{2d} K_{d,d}, \frac{n}{d+1} K_{d+1}$ is the maximizer of the number of homomorphisms for all $G$. In \cite{zhao2}, Zhao gave an infinite family of $H$ for which $\hom(G,H) \leq \hom(K_{d,d},H)^{n/2d}$ for all $n$-vertex, $d$-regular $G$, and in \cite{mainConjReference} Galvin exhibited infinitely many triples $(n,d,H)$ for which $\hom(G,H) \leq \hom(K_{d+1},H)^{n/(d+1)}$ for all $n$-vertex, $d$-regular G. The second aim of this article is to continue this line of investigation. We develop the properties of two graph operations, and then use them to establish a large family of choices of $H$ for which some number of copies of $K_{d,d}$ is the maximizer, as well as another large family for which an appropriate number of copies of $K_{d+1}$ is the maximizer.

Specifically, we prove the following.

\begin{thm}\label{fullthm}
Let $H_1$ and $H_2$ be graphs with the property that for all n and d and all $n$-vertex, $d$-regular $G$, we have $\hom(G,H_i) \leq \hom(K_{d,d},H_i)^{n/2d}$. Let $A$ be an arbitrary graph, and $B$ an arbitrary bipartite graph. If $\tilde{H}$ is any of $H_1 \times H_2$, $H_1^A$, or $A^B$, then for all $n$ and $d$ and all $n$-vertex, $d$-regular $G$, we have
\begin{equation}
\hom(G,\tilde{H}) \leq \hom(K_{d,d},\tilde{H})^{n/2d}. \label{eqnfromfullthm}
\end{equation}
\end{thm}

The operations of tensor product $H_1 \times H_2$ and graph exponentiation $A^B$ will be defined formally in Section \ref{graphOperationsSection}.

As an example of Theorem \ref{fullthm}, consider the graph $H$ consisting of one edge, with one endvertex looped. For this choice of $H$, it is easy to see that $\hom(G,H)$ is the number of independent sets in $G$ (that is, the number of sets of mutually non-adjacent vertices). Zhao \cite{Zhao} showed that in this case $\hom(G,H) \leq \hom(K_{d,d}, H)^{\frac{n}{2d}}$ holds for all $G$. Now applying Theorem \ref{fullthm} we find that (\ref{eqnfromfullthm}) holds for all $n$-vertex $d$-regular $G$ with $\tilde{H} = H \times H$. This graph turns out the be the graph on four vertices composed of a triangle with one vertex looped, and an additional vertex adjacent to the looped vertex.

It's worth noting that this is not a graph that can be addressed using the methods Zhao develops in \cite{zhao2}. Zhao provides the following sufficient condition.

\begin{thm}\label{zhaoscriterion}
For an arbitrary graph $H$, construct a graph $H^{\rm bst}$ with vertices $V(H^{\rm bst}) = V(H) \times V(H)$ and an edge between $(u, v)$ and $(u', v') \in V(H^{\rm bst})$ if and only if $uu' \in E(H)$, $vv' \in E(H)$, and either $uv' \notin E(H)$ or $u'v \notin E(H)$. If $H^{\rm bst}$ is bipartite, then for all $n$ and $d$ and all $n$-vertex, $d$-regular $G$,
\begin{equation}
\hom(G,H) \leq \hom(K_{d,d},H)^{n/2d}.
\end{equation}
\end{thm}

Our choice of $\tilde{H}$ does not meet this criterion, because $\tilde{H}^{\rm bst}$ contains an odd cycle (the construction is straightforward, and has been omitted).

We also make use of the tensor product and the graph exponent to establish a family of choices of $H$ for which $\hom(G, H)$ is upper bounded by $\hom(K_{d+1}, H)^{\frac{n}{d+1}}$.

\begin{thm}\label{widomAndFriends}
Let $G$ be an arbitrary $d$-regular graph on $n$ vertices. For an arbitrary graph $H$ and bipartite graph $B$,
\[\hom(G, l(H^B)) \leq \hom(K_{d+1}, l(H^B))^{\frac{n}{d+1}}\]
where $l(G)$ denotes the induced subgraph of $G$ formed by the set of looped vertices of $G$.
\end{thm}

When $H$ is a path on two vertices with a loop on the first, and $B$ is a single edge, $l(H^B)$ is easily seen to be a path on three vertices with a loop at each vertex. This graph, which we will call $H_{\rm WR}$, is often referred to as the {\em Widom-Rowlinson} graph and it encodes a well-studied model in statistical physics (see \cite{winkler} for a good discussion). Galvin conjectured (implicitly in \cite{mainConjReference} and explicitly in \cite{Kalamazoo}) that ${\hom}(G,H_{\rm WR}) \leq {\hom}(K_{d+1},H_{\rm WR})^{n/(d+1)}$ for all $n$-vertex $d$-regular $G$, making it the simplest non-trivial example of an $H$ for which unions of $K_{d+1}$ maximize the homomorphism count among regular graphs, and Theorem \ref{widomAndFriends} verifies this conjecture.

In an earlier version of this paper (arXiv:1510.01833v1) the right-hand side of the inequality in Theorem 3 was the weaker ${\hom}(K_d,l(H^B))^{n/d}$. This turned out to be the result of a small error in the proof. We are grateful to E. Cohen (personal communication) for pointing out this error. After the earlier version of this paper appeared, a separate result that also implies Galvin's conjecture was obtained by Cohen, Perkins and Tetali \cite{CPT}.

In Section \ref{frSection} we recall some results that will be important in what follows. In Section \ref{counterexampleSection} we give a construction providing counterexamples to Conjecture \ref{fullcnj}. The results from Section \ref{frSection} are then built upon in Section \ref{graphOperationsSection}, culminating in the proof of Theorem \ref{fullthm}. Lastly, in Section \ref{widomSection}, we continue the discussion of exponential graphs in order to prove Theorem \ref{widomAndFriends}.

\section{Foundational Results} \label{frSection}
Before delving into the main body of work, it will be useful to recall a handful of important results that we will draw on heavily. First, we will make extensive use of a theorem due to Lov\'asz \cite{Lovasz}, that makes it very easy to derive relationships between graphs. It does this by allowing arguments about graph homomorphisms to be used in reasoning about the graphs themselves.

\begin{thm}\label{homoToIso}
Let $G_1$ and $G_2$ (possibly with loops) be given. Suppose that for all choices of $H$,
\begin{equation}
\hom(G_1, H) = \hom(G_2, H)
\end{equation}
then $G_1 = G_2$.

Likewise, if $H_1$ and $H_2$ are given, and for all connected $G$ we have
\begin{equation}
\hom(G, H_1) = \hom(G, H_2)\label{secondHomoToIso}
\end{equation}
then $H_1 = H_2$.
\end{thm}
Here, as elsewhere, we take $G_1 = G_2$ to mean that the graphs are isomorphic.

Lov\'asz proves the slightly weaker statement where \eqref{secondHomoToIso} is required to hold for all $G$ (including graphs which are not connected), but this can be strengthened easily by noting that if a graph $G$ is not connected, then it can be written as the disjoint union of its components, $G_1 \cup G_2 \cup ... \cup G_k$. These can be $H$-colored (assigned a homomorphism to $H$) independently, and so
\[\hom(G, H) = \hom(G_1, H)\hom(G_2, H)\ldots\hom(G_k, H).\]
Thus, if \eqref{secondHomoToIso} holds for all connected graphs $G$, it follows that it holds for all graphs.

Second, much of the discussion that follows will be concerned with bipartite graphs, and so it will be useful to have some characterization of them that can be expressed in the language of homomorphisms. The following lemma provides this.

\begin{lem}\label{bipartiteCharacterization}
 A graph $H$ is bipartite exactly when for every non-bipartite $G$ (including those with loops), $\hom(G, H) = 0$.
\end{lem}
The proof is straightforward, and has been omitted.

\section{Counterexamples to Conjecture \ref{fullcnj}}\label{counterexampleSection}

As was alluded to in the introduction, we have been able to show that Conjecture \ref{fullcnj} is false. We present a construction of an infinite family of counterexamples below. Recently, a similar construction was independently found by P. Devlin (personal communication from J. Kahn).

\begin{thm}\label{counterexampleMain}
Let $d$ be given. Let $H$ be any simple (unlooped) graph with no $(d+1)$-clique. Let $G$ be any connected $d$-regular graph on $n < 2d$ vertices such that $\hom(G,H) > 0$. (Note that this implies that $G$ is unlooped). Then there exists a natural number $k$ such that
\[\hom(G, kH) > \hom(K_{d+1}, kH)^{\frac{n}{d+1}}\]
and
\[\hom(G, kH) > \hom(K_{d,d}, kH)^{\frac{n}{2d}}.\]
\end{thm}

\begin{proof}
The first inequality is straightforward. Since $H$ has no $(d+1)$-clique, $\hom(K_{d+1}, kH) = 0$. Thus,
\[\hom(G, kH) > \hom(K_{d+1}, kH)^{\frac{n}{d+1}}\]
trivially.

The second inequality follows from the observation that
\[\hom(G, kH) = k \hom(G, H)\]
and
\[\hom(K_{d,d}, kH) = k\hom(K_{d,d}, H).\]
Applying these, it is sufficient to show that
\[k\hom(G, H) > (k\hom(K_{d,d}, H))^{\frac{n}{2d}}\]
or, rearranging,
\[k^{1-\frac{n}{2d}} > \frac{\hom(K_{d,d}, H)^{\frac{n}{2d}}}{\hom(G, H)}\]
Since $n < 2d$, we know that $1-\frac{n}{2d} > 0$, and so the left side increases with increasing $k$. Choosing
\begin{equation}\label{valueofK}
k > \left(\frac{\hom(K_{d,d}, H)^{\frac{n}{2d}}}{\hom(G, H)}\right)^{\frac{1}{1-\frac{n}{2d}}}
\end{equation}
gives the desired result.
\end{proof}

In order to verify that counterexamples exist for a particular choice of $d$, it suffices to find $G$ and $H$ satisfying the conditions of Theorem \ref{counterexampleMain}. One very natural choice of $H$ is $K_d$. For this choice of $H$, and for $d > 6$, one possible choice for $G$ is $C_{d-2} + C_{d-2}$, where $+$ denotes the join operation (the join of two graphs $G_1$ and $G_2$ is just their disjoint union, together with all possible edges from any vertex of $G_1$ to any vertex of $G_2$). This is $q$-colorable for all $q \geq 6$, since we can color one copy of $C_{d-2}$ with three colors, and the other copy with the other three. It is $d$-regular, and has fewer than $2d$ vertices. Thus, Theorem \ref{counterexampleMain} gives us that there is a $k$ such that $k K_d$ is a counterexample to Conjecture \ref{fullcnj}. (As an aside, in P. Devlin's construction, he provided the useful observation that Brooks' theorem gurantees that any connected $d$-regular graph $G$ on $n$ vertices other than $K_{d+1}$ will satisfy these conditions, so choices of $G$ that will lead to counterexamples are plentiful).

One point that should be raised about the construction in Theorem \ref{counterexampleMain} is that $kH$ is not connected unless $k = 1$ and $H$ is connected. While connectedness is not a requirement, it might seem plausible that we could save some form of Conjecture \ref{fullcnj} by restricting our attention to connected graphs. Unfortunately, this is not the case. Choosing $G$ and $H$ as in Theorem \ref{counterexampleMain}, with $H$ connected, it is possible to construct a connected $H'$ such that Conjecture \ref{fullcnj} fails.

To see this, let $H'$ be $k$ copies of $H$ joined by paths of length $2d$. For convenience, also define $H''$ to be one such copy of $H$, together with all of the paths emanating from that copy (for this to be well-defined, we require that the paths are joined to each copy of $H$ in the same way, in the sense that for every pair of copies of $H$, there is an automorphism of $H'$ that exchanges them). Every $H'$-coloring of $K_{d,d}$ lies completely in one of the $k$ copies of $H$ together with the paths emanating from it, so
\[\hom(K_{d,d}, H') \leq \hom(K_{d,d}, k H'').\]

On the other hand, it is clear that
\[\hom(G, H') \geq \hom(K_{d,d}, k H)\]
since $k H$ is a subgraph of $H'$. This means that
\[\frac{\hom(K_{d,d}, k H'')^{\frac{n}{2d}}}{\hom(G, k H)} \geq \frac{\hom(K_{d,d}, H')^{\frac{n}{2d}}}{\hom(G, H')}\]
but
\[\frac{\hom(K_{d,d}, k H'')^{\frac{n}{2d}}}{\hom(G, k H)} = k^{\frac{n}{2d} - 1}\frac{\hom(K_{d,d}, H'')^{\frac{n}{2d}}}{\hom(G, H)}\]
which is less than one for sufficiently large $k$ (recalling that $n < 2d$). For such a $k$, then, we have
\[\hom(G, H) > \hom(K_{d,d}, H')^{\frac{n}{2d}}\]
as before.

So it seems that any general conjecture will need to differ substantively from Conjecture \ref{fullcnj}. Because the counterexamples given here provide an $H$ given a particular $d$, one possibility is that for each $H$, Conjecture \ref{fullcnj} holds for sufficiently large $d$. Another potential fix would be to weaken the conjecture, positing merely that there is some finite set of maximizer graphs, which includes $K_{d+1}$, $K_{d,d}$, and some finite list of other possibilities. But as the conjecture evolves to accommodate these new counterexamples, the price is often paid in strength or in elegance. Demonstrating results that approach the general case is not easy, and it seems that even making an effective conjecture is a difficult task.

\section{Graph Operations}\label{graphOperationsSection}

Despite the setbacks described in Section \ref{counterexampleSection}, Question \ref{mainQuestion} can be answered for certain choices of $H$. The strategy presented here builds on the following result, first explored in a limited way in \cite{kahn} and later extended by Galvin and Tetali in \cite{bipartiteReference} into its current form.

\begin{thm}\label{bipartite}
Let $G$ be a $d$-regular bipartite graph on $n$ vertices. Then for all graphs $H$, $\hom(G,H) \leq \hom(K_{d,d},H)^{\frac{n}{2d}}$.
\end{thm}

This theorem is tantalizingly general in its restrictions on $H$, but it requires that $G$ be bipartite, a significant limitation. Nonetheless, for certain choices of $H$, it is possible to remove this restriction. As we have mentioned Zhao was able to show this for one particular choice of $H$ \cite{Zhao}, and he later generalized his results to a larger family \cite{zhao2}. We present methods which allow us extend his work to a much larger family of graphs. To describe these methods we will need the notion of a tensor product.

\subsection{The Tensor Product}\label{tp}
\begin{defn}
The tensor product of graphs $G_1$ and $G_2$, denoted $G_1 \times G_2$ is the graph whose vertex set is the Cartesian product $V(G_1) \times V(G_2)$, with an edge between $(g_1, g_2)$ and $(g_1', g_2')$ if and only if $g_1 g_1'$ is an edge of $G_1$ and $g_2 g_2'$ is and edge of $G_2$.
\end{defn}

If we think of $G_1$ and $G_2$ as adjacency matrices, this tensor product is the same as the tensor product from linear algebra, which can be seen by simply writing out the definitions (we leave this as an exercise).

Tensor products have many nice properties, including commutativity, associativity, and distributivity with respect to graph unions. It is the tensor product's relationship to homomorphisms that primarily concerns us here, although we will be able to use this property to prove many others.

The following result exists in the literature, but its origins are unclear. One statement is given in \cite{Lovasz}.

\begin{thm}\label{tensormult}
Let $G$, $H_1$ and $H_2$ be arbitrary. Then
\[\hom(G, H_1 \times H_2) = \hom(G, H_1)\hom(G, H_2).\]
\end{thm}

This property is useful in that it gives the relationship between tensor products and graph homomorphisms. Used in conjunction with Theorem \ref{homoToIso}, it offers a very succinct way of proving many of the important properties of tensor products. Many of these exist in the literature (one example is in \cite{hell}), but they will be useful in the results to come, and their proofs provide an illustration of the utility of Theorem \ref{homoToIso}, which will be indispensable later on. In the discussion below, let $l_k$ be the graph on $k$ vertices in which every vertex has a loop and there are no other edges.

\begin{thm}[Properties of the Tensor Product]
For arbitrary graphs $A$, $B$ and $C$,
\begin{align}
&A \times l_1 = A \tag{Identity}\\
&A \times B = B \times A \tag{Commutativity}\\
&A \times (B \times C) = (A \times B) \times C \tag{Associativity}\\
&A \times (B \cup C) = (A \times B) \cup (A \times C) \tag{Distributivity}\\
&kA = A \times l_k
\end{align}

Furthermore, $A \times B$ is bipartite if and only if at least one of $A$ or $B$ is bipartite.
\end{thm}

\begin{proof}
The proof that $l_1$ is the identity follows immediately from the definition. The proofs of most of the remaining parts use Theorem \ref{homoToIso} and Theorem \ref{tensormult} almost exclusively.

To prove commutativity, we have by Theorem \ref{tensormult} that for all $G$
\begin{align*}
\hom(G, A \times B) = &\hom(G, A)\hom(G, B)\\
= &\hom(G, B)\hom(G, A)\\
= &\hom(G, B \times A).
\end{align*}
By Theorem \ref{homoToIso}, then, $A \times B = B \times A$.

The proofs of associativity and distributivity with respect to graph unions follow essentially the same scheme, using Theorem \ref{homoToIso} to reduce the problem to reasoning about integers. Distributivity across graph unions requires that we look at all connected $G$, rather than all $G$, but this is enough for Theorem \ref{homoToIso} to work, so this concession will cost us nothing. Distributivity also implies that $kA = A \times l_k$, working from the fact that $A \times l_1 = A$. We omit the specifics of the proofs for brevity.

Lastly, to determine when $A \times B$ is bipartite, observe that for all $G$, $\hom(G, A \times B) = 0$ exactly when $\hom(G, A) \hom(G, B) = 0$, i.e. when either $\hom(G, A) = 0$ or $\hom(G, B) = 0$. Thus, $A \times B$ is bipartite when at least one of $A$ or $B$ is bipartite, by Lemma \ref{bipartiteCharacterization}.
\end{proof}

The tensor product's relationship to bipartite graphs offers one possible strategy for extending Theorem \ref{bipartite} to non-bipartite $G$. The idea is one that was first pioneered in a less general context by Zhao in \cite{Zhao}. Zhao wasn't working with tensor products directly, but rather with a graph operation called the bipartite double cover. Nonetheless, the bipartite double cover of a graph can be expressed as a tensor product, and this is how we will define it here.

\begin{defn}
The bipartite double cover of a graph $G$ is given by $G \times K_2$. Explicitly, the bipartite double cover is constructed on two copies of the vertex set of $G$, $V_1$ and $V_2$. Place an edge between a vertex $u$ in $V_1$ and a vertex $v$ in $V_2$ if and only if there is an edge between the corresponding vertices of $G$.
\end{defn}

The double cover of a graph has many important attributes, but the one that is most relevant to us here is apparent from giving its definition as a tensor product-- it is always bipartite. Thus, by Theorem $\ref{bipartite}$, for any $d$-regular graph $G$ on $n$ vertices we have that
\[\hom(G\times K_2,H) \leq \hom(K_{d,d},H)^{\frac{n}{d}}.\]
(The exponent is $\frac{n}{d}$ instead of $\frac{n}{2d}$ because $G\times K_2$ has $2n$ vertices.) So if we could simply show that
\[\hom(G,H)^2 \leq \hom(G\times K_2,H),\]
we would have that
\[\hom(G,H) \leq \hom(K_{d,d},H)^{\frac{n}{2d}}.\]
Thus, it is of great interest to know which graphs $H$ satisfy
\[\hom(G,H)^2 \leq \hom(G\times K_2,H)\]
for all $G$.

This motivates the following definition.

\begin{defn}
A graph $H$ is bipartite reducible if for all $G$,
\begin{equation}
\hom(G,H)^2 \leq \hom(G\times K_2,H). \label{hstar}
\end{equation}
Here and throughout, the set of all bipartite reducible $H$ is denoted $\mathcal{H}$.
\end{defn}

As discussed, the property that makes this set interesting here is the following, first observed in \cite{Zhao} (Zhao uses the name ``strongly GT'' rather than ``bipartite reducible'').
\begin{thm}\label{bipRed}
If $H$ is bipartite reducible, then for all $d$-regular graphs $G$ on $n$ vertices,
\[\hom(G,H) \leq \hom(K_{d,d},H)^{\frac{n}{2d}}.\]
\end{thm}
\begin{proof}
We have that
\[\hom(G,H)^2 \leq \hom(G\times K_2,H) \leq \hom(K_{d,d},H)^{\frac{n}{d}}.\]
So
\[\hom(G,H) \leq \hom(K_{d,d},H)^{\frac{n}{2d}}.\]
\end{proof}

Zhao was able to show that the independent set graph (a single edge with one vertex looped) is in $\mathcal{H}$, and in \cite{zhao2} he extended this to a larger family of graphs. In addition, any bipartite graph $H$ is clearly in $\mathcal{H}$, since for nonbipartite $G$, a bipartite $H$ will give $\hom(G, H) = 0$.

But it is possible to move beyond these examples. Using the properties derived above, the following result is straightforward.

\begin{thm}\label{tensorClosure}
$\mathcal{H}$ is closed under tensor multiplication.
\end{thm}
\begin{proof}
Simply note that for arbitrary graphs $A,B \in \mathcal{H}$, we have
\begin{align*}
\hom(G, A \times B)^2 = &\left(\hom(G,A)\hom(G,B)\right)^2\\
= &\hom(G,A)^2\hom(G,B)^2\\
\leq &\hom(G\times K_2,A)\hom(G \times K_2,B)\\
= &\hom(G\times K_2,A \times B).
\end{align*}
\end{proof}

Closure under tensor multiplication allows us to put many graphs into $\mathcal{H}$ that were not accessible by previous methods. One simple example that was discussed already is $H \times H$, where $H$ is a graph consisting of a single edge with a loop on one endvertex. $H \times H$ is non-bipartite. and cannot be shown to be in $\mathcal{H}$ using Zhao's result alone. Many other permutations and combinations of graphs in $\mathcal{H}$ exist, allowing us to categorize a large number of graphs.

\subsection{Graph Exponentiation}

We are also able to show that $\mathcal{H}$ is closed under another operation, one that will allow us to reshape the definition itself into something that looks very different from what it is now. This operation is not widely discussed in the literature, so we will take some time to explore it in detail here.

\begin{defn}
For any two graphs $A$ and $B$, the $B$ graph exponent of $A$, denoted $A^B$, is a graph on the set of all functions from the vertices of $B$ to the vertices of $A$, with an edge placed between two functions $f_1$ and $f_2$ if for every edge $uv$ in $B$, $f_1(u)f_2(v)$ is an edge of $A$.
\end{defn}

At first glance, this definition may seem somewhat difficult to work with, and getting an intuitive grasp for what the result of this operation should be is not easy. But just as with the tensor product, we can establish a relationship between this operation and graph homomorphisms that will allow us not only to derive a wide array of properties of this operation, but also to establish its relationship to $\mathcal{H}$.

\begin{thm}\label{exponents}
For any graphs $G$, $H$, and $A$,
\[\hom(G \times A, H) = \hom(G, H^A)\]
\end{thm}
\begin{proof}
Let $\phi$ be a homomorphism from $G \times A$ to $H$. Let $\phi_g: V\left(A\right) \rightarrow V\left(H\right)$ be the function given by $\phi_g\left(a\right) = \phi\left(\left(g, a\right)\right)$. That is, $\phi_g$ is just $\phi$ after fixing a vertex $g \in V\left(G\right)$. Define $\psi$ to be the function on the vertices of $G$ given by $\psi\left(g\right) = \phi_g$. Then $\psi$ is a function from the vertices of $G$ to the vertices of $H^A$ (recalling that those vertices are themselves functions). It preserves adjacency, because if $g_1g_2$ is an edge of $G$, then for all edges $a_1a_2$ of $A$, $\left(g_1, a_1\right)\left(g_2, a_2\right)$ is an edge of $G \times A$, and so $\phi\left(\left(g_1, a_1\right)\right)\phi\left(\left(g_2, a_2\right)\right)$ will be an edge of $H$. Therefore, $\psi$ is a homomorphism. This reasoning can also be done in reverse, so there is a one-to-one correspondence between homomorphisms from $G \times A$ to $H$ and homomorphisms from $G$ to $H^A$, i.e.
\[\hom(G \times A, H) = \hom(G, H^A)\]
\end{proof}

As in the case of tensor products, this relationship with graph homomorphisms will allow us to develop a rich set of properties for this operation. We prove a number of these below, some of which serve to motivate thinking about this operation as an exponent.

\begin{thm}[Properties of Graph Exponents]
For arbitrary graphs $A$, $B$ and $C$,
\begin{align}
&A^{l_1} = A \tag{Identity}\\
&(A^B)^C = A^{B \times C}\\
&(A \times B)^C = A^C \times B^C \tag{Distributivity}\\
&A^B \times A^C = A^{B \cup C}\\
&\underbrace{A \times A \times \ldots \times A}_{k \textnormal{ times}} = A^{l_k}\label{repeatedK}\\
&(A \cup B)^C = A^C \cup B^C \cup E_n \tag{Freshman's Dream}\\
&\hom(A,B^C) = \hom (C, B^A)
\end{align}
In freshman's dream $n$ is chosen such that $(A \cup B)^C$ and $A^C \cup B^C \cup E_n$ have the same number of vertices, and $C$ is connected. Furthermore, for non-empty $A$, we show that $A^B$ is bipartite exactly when $A$ is bipartite and $B$ is not.
\end{thm}
\begin{proof}
The proof of most parts will resemble those given for tensor products in Section \ref{tp}. As before, we first establish the desired property in the context of homomorphisms using the relationship we obtained (in this case in Theorem \ref{exponents}), and then use Theorem \ref{homoToIso} to remove the dependence on homomorphisms.

It is easy to see that the identity result is just (\ref{repeatedK}) with $k = 1$. We include it simply because it demonstrates the existence of an identity element for this operation.

The first substantive statement to be proved is $(A^B)^C = A^{B \times C}$. For all $G$,
\begin{align*}
\hom(G, (A^B)^C) = &\hom(G \times C, A^B)\\
= &\hom(G \times C \times B, A)\\
= &\hom(G \times (B \times C), A)\\
= &\hom(G, A^{B \times C}).
\end{align*}
By Theorem \ref{homoToIso}, this implies that $(A^B)^C = A^{B \times C}$.

The same sort of reasoning gives us that $(A \times B)^C = A^C \times B^C$. For all $G$,
\begin{align*}
\hom(G, (A \times B)^C) = &\hom(G \times C, A \times B)\\
= &\hom(G \times C, A)\hom(G \times C, B)\\
= &\hom(G, A^C)\hom(G, B^C)\\
= &\hom(G, A^C \times B^C),
\end{align*}
and so, using as always Theorem \ref{homoToIso}, this means that $(A \times B)^C = A^C \times B^C$.

To demonstrate the claim that $A^B \times A^C = A^{B \cup C}$, observe that for all $G$ we have
\begin{align*}
\hom(G, A^B \times A^C) = &\hom(G, A^B)\hom(G, A^C)\\
= &\hom(G\times B, A)\hom(G \times C, A)\\
= &\hom((G\times B) \cup (G \times C), A)\\
= &\hom(G\times(B \cup C), A)\\
= &\hom(G, A^{B \cup C})
\end{align*}
which implies that $A^B \times A^C = A^{B \cup C}$.

To prove that $A \times A \times \cdots \times A = A^{l_k}$, note that for all $G$,
\begin{align*}
\hom(G, A \times A \times \cdots \times A) = &\hom(G, A)^k\\
= &\hom(kG, A)\\
= &\hom(G \times l_k, A)\\
= &\hom(G, A^{l_k})
\end{align*}
and so $A \times A \times \cdots \times A = A^{l_k}$.

Freshman's dream will require just a little more work. Let $G$ be connected, and this time, let's also assume $G$ is non-empty. Then for any natural number $n$,
\begin{align*}
\hom(G, (A \cup B)^C) = &\hom(G \times C, A \cup B)\\
= &\hom(G \times C, A)+\hom(G \times C, B)\\
= &\hom(G, A^C)+\hom(G, B^C)\\
= &\hom(G, A^C \cup B^C)\\
= &\hom(G, A^C \cup B^C \cup E_n),
\end{align*}
where the last step follows because $G$ is non-empty, and so the image of any homomorphism from $G$ cannot contain an isolated vertex, meaning we can add isolated vertices at will.

To make use of Theorem \ref{homoToIso}, it still remains to find a value of $n$ such that this expression will hold when $G$ is empty. But if $G$ is both empty and connected, it is a single vertex, and so $\hom(G, H)$ just counts the vertices of $H$. So if we choose $n$ such that $(A \cup B)^C$ has the same number of vertices as $A^C \cup B^C \cup E_n$, the number of homomorphisms from the singleton graph will be the same for each. Now we have equality for all connected $G$, and so Theorem \ref{homoToIso} gives us that
\[(A \cup B)^C = A^C \cup B^C \cup E_n\]

The claim that $\hom(A,B^C) = \hom(C, B^A)$ holds because
\begin{align*}
\hom(A,B^C) = &\hom(A \times C, B)\\
= &\hom(C \times A, B)\\
= &\hom(C, B^A).
\end{align*}

Lastly, the claim that $A^B$ is bipartite for non-empty $A$ exactly when $A$ is bipartite and $B$ is not can be shown as follows. Assume first that $B$ is bipartite. Then for all graphs $G$, $G \times B$ is bipartite. Since $A$ is non-empty, it has at least one edge which can be used to color any bipartite graph. Thus,
\[\hom(G, A^B) = \hom(G \times B, A) > 0\]
and so by Lemma \ref{bipartiteCharacterization}, $A^B$ is non-bipartite. Now suppose $B$ is non-bipartite. Then $G \times B$ is bipartite exactly when $G$ is bipartite. Suppose $A$ is bipartite. Then for all non-bipartite $G$,
\[\hom(G, A) = 0.\]
But whenever $G$ is non-bipartite so is $G \times B$, and so
\[\hom(G, A^B) = \hom(G \times B, A) = 0\]
which by Lemma \ref{bipartiteCharacterization} means that $A^B$ is bipartite.
On the other hand, suppose $A$ is non-bipartite. Then there exists a non-bipartite $G$ such that
\[\hom(G, A) > 0.\]
Notice that an $A$-coloring of $G$ gives rise to an $A$-coloring of $G \times B$, since we can color each vertex of $G \times B$ with the color of the underlying vertex from $G$. Thus,
\[\hom(G, A^B) = \hom(G \times B, A) > 0\]
and so $A^B$ is non-bipartite.
\end{proof}

One immediate application of these results is that the criterion for a graph $H$ to be in $\mathcal{H}$ can be re-framed as the statement that for all $G$
\[\hom(G, H^{l_2}) \leq \hom(G, H^{K_2})\]
where $l_2$ is the graph on two vertices with a loop on each vertex. While none of our results make use of this form preferentially over the statement given in (\ref{hstar}), it has the appeal that $G$ now appears by itself in each homomorphism, which could prove useful for certain kinds of arguments (particularly inductions).

This kind of re-framing of homomorphism problems is something that graph exponentiation does quite well. In fact, graph exponentiation is able to capture all of the behavior of graph homomorphisms, in the following sense.

\begin{obs}
Let $G$ and $H$ be arbitrary. Then $\hom(G, H)$ is equal to the number of loops in $H^G$.
\end{obs}

To see why this is true, let $f$ be a looped vertex of $H^G$. Then by definition, for all edges $g_1 g_2$ of $G$, $f(g_1)f(g_2)$ is an edge of $H$. So the looped vertices of $H^G$ are exactly the homomorphisms from $G$ to $H$. This suggests a more radical restatement of (\ref{hstar}). For a graph $H$ to be in $\mathcal{H}$, we simply require that for all $G$, $H^{G \times k_2}$ has at least as many loops as $H^{G \times l_2}$.

The properties of graph exponentiation also make it easy to establish its relationship to $\mathcal{H}$.

\begin{thm}\label{expClosure}
Let $A$ be an arbitrary graph. Then $\mathcal{H}$ is closed under $A$-exponentiation. That is, for any graph $H \in \mathcal{H}$, $H^A \in \mathcal{H}$. Note that this is strictly more general than saying $\mathcal{H}$ is closed under graph exponentiation, since $A$ need not be in $\mathcal{H}$.
\end{thm}

\begin{proof}
Let $H \in \mathcal{H}$ be given. Then
\[\hom(G,H)^2 \leq \hom(G \times K_2, H)\]
for all $G$. Thus, in particular, it holds for $G' = G \times A$. So
\begin{align*}
\hom(G,H^A)^2 = &\hom(G \times A,H)^2\\
\leq &\hom((G \times A) \times K_2, H)\\
= &\hom(G \times K_2 \times A, H)\\
= &\hom(G \times K_2, H^A)
\end{align*}
as desired.
\end{proof}

This is enough to give the promised relationship to $\mathcal{H}$, but graph exponentiation has a second mode of interaction.

\begin{thm}\label{bipartiteClosure}
Let $A$ be an arbitrary bipartite graph. Then for all graphs $H$  (not necessarily with $H \in \mathcal{H}$), $H^A \in \mathcal{H}$.
\end{thm}
\begin{proof}
Notice that when $G$ is bipartite, $G \times l_2 = G \times K_2$ (we leave this as an exercise for now, but a proof is provided in Section \ref{widomSection} as Lemma \ref{l2toK2}). Since $A$ is bipartite, $G \times A$ is bipartite, and so
\[G \times A \times l_2 = G \times A \times K_2\]
since the tensor product of a bipartite graph with $l_2$ is the same as its product with $K_2$, up to isomorphism. But then
\begin{align*}
\hom(G, H^A)^2 = &\hom(G \times A, H)^2\\
= &\hom(G \times A \times l_2, H)\\
= &\hom(G \times A \times K_2, H)\\
= &\hom(G \times K_2, H^A).
\end{align*}
So in this case the desired inequality is always tight.
\end{proof}

Notice also that this was the last case in Theorem \ref{fullthm} to be proven, and so together with Theorem \ref{tensorClosure} and Theorem \ref{expClosure}, we have proven all of the cases of Theorem \ref{fullthm}.

\section{Exponential graphs and the ``other" maximizer}\label{widomSection}
Conjecture \ref{fullcnj}, while false, has provided valuable intuition about Question \ref{mainQuestion}. So far we have discussed one of the upper bounds it posited, namely $\hom(K_{d,d},H)^{\frac{n}{2d}}$, but what about the other one?

As it turns out, many of the ideas we have developed in the preceeding sections are applicable to finding graphs on this side of the problem as well. With all of the tools we now have, it is possible to prove Theorem \ref{widomAndFriends}, which provides a clear maximizer for a large family of graphs, most notably the Widom-Rowlinson graph.

The result will require two small lemmas. Let $G^{\circ}$ denote the graph formed by placing a loop on every vertex in $G$, and let $l(G)$ denote the induced subgraph of $G$ formed by the set of looped vertices of $G$.

\begin{lem}\label{loops}
For any graphs $G$ and $H$,
\[\hom(G, l(H)) = \hom(G^{\circ}, H).\]
\end{lem}
\begin{proof}
Note that
\[\hom(G, l(H)) = \hom(G^{\circ}, l(H)),\]
since the added loops only impose a restriction on unlooped vertices of $l(H)$, of which there are none. But it is also true that
\[\hom(G^{\circ}, l(H)) = \hom(G^{\circ}, H)\]
since no vertex of $G^{\circ}$ is permitted to map to the unlooped vertices of $H$.
Thus,
\[\hom(G, l(H)) = \hom(G^{\circ}, H).\]
\end{proof}

The second lemma can be thought of as another property of tensor multiplication. We have made use of the result already, in proving Theorem \ref{bipartiteClosure}.

\begin{lem}\label{l2toK2}
For any bipartite graph $B$, $B \times K_2 = B \times l_2$.
\end{lem}
\begin{proof}
Since $B$ is bipartite, it can be bipartitioned into some $B_1$ and $B_2$. Recalling that the tensor product of $B$ and $K_2$ is defined on vertex set $V(B) \times V(K_2)$, in order for there to be an edge between two vertices in $B \times K_2$, they must contain vertices from $B$ that are from different bipartitions, and they must not both contain the same vertex of $K_2$. Partition the vertices of $B \times K_2$ into four classes, based on which bipartition the vertex from $B$ belongs to, and which vertex of $K_2$ is used. Call the partitions $(B_1, k_1),(B_1, k_2),(B_2, k_1),$ and $(B_2, k_2)$. Then vertices in $(B_1, k_1)$ can be adjacent to vertices in $(B_2, k_2)$, and vertices in $(B_1, k_2)$ can be adjacent to vertices in $(B_2, k_1)$, but these will be the only edges present. Notice that this is just two copies of $B$, and so
$B \times K_2 = B \cup B = B \times l_2$.
\end{proof}

This is enough to give a proof of Theorem \ref{widomAndFriends}, the main result of this section.

\begin{thm*}[Restatement of Theorem \ref{widomAndFriends}]
Let $G$ be an arbitrary $d$-regular graph on $n$ vertices. For an arbitrary graph $H$ and bipartite graph $B$,
\[\hom(G, l(H^B)) \leq \hom(K_{d+1}, l(H^B))^{\frac{n}{d+1}}.\]
\end{thm*}
\begin{proof}
Notice that $K_{d,d}$ can be written as $K_d^{\circ} \times K_2$. This can be used as follows.
\begin{align*}
\hom(G, l(H^B)) &= \hom(G^{\circ},H^B)\tag{Lemma \ref{loops}}\\
&= \hom(G^{\circ} \times K_2,H^B)^{\frac{1}{2}}\tag{Theorem \ref{bipartiteClosure}}\\
&\leq \hom(K_{d+1, d+1},H^B)^{\frac{n}{2(d+1)}}\tag{Theorem \ref{bipartite}}\\
&= \hom(K_{d+1}^{\circ} \times K_2,H^B)^{\frac{n}{2(d+1)}}\\
&= \hom(K_{d+1}^{\circ},\left(H^B\right)^{K_2})^{\frac{n}{2(d+1)}}\\
&= \hom(K_{d+1}^{\circ},H^{B \times K_2})^{\frac{n}{2(d+1)}}\\
&= \hom(K_{d+1}^{\circ},H^{B \times l_2})^{\frac{n}{2(d+1)}}\tag{Lemma \ref{l2toK2}}\\
&= \hom(K_{d+1}^{\circ},\left(H^B\right)^{l_2})^{\frac{n}{2(d+1)}}\\
&= \hom(K_{d+1}^{\circ} \times l_2,H^B)^{\frac{n}{2(d+1)}}\\
&= \hom(K_{d+1}^{\circ},H^B)^{\frac{n}{d+1}}\\
&= \hom(K_{d+1},l(H^B))^{\frac{n}{d+1}}.\tag{Lemma \ref{loops}}\\
\end{align*}
One important detail here is that $G^{\circ} \times K_2$ is actually ${d+1}$-regular, since the loops contribute once to each copy of the vertex set of $G$. (In an earlier version of this paper $G^{\circ} \times K_2$ was erroneously treated as being $d$-regular; this was the error pointed out to us by E. Cohen.)
\end{proof}

This result is interesting in that it gives some number of copies of $K_{d+1}$ as the maximizer for $l(H^B)$, while we know from Theorem \ref{bipartiteClosure} that for $H^B$ the maximizer is some number of copies of $K_{d,d}$. That is, removing the unlooped vertices of $H$ is enough to change the maximizer.

\end{document}